\documentclass[12pt,a4paper]{article}
\usepackage[top=1in, bottom=1in, left=1in, right=1in]{geometry}
\usepackage{amsmath, amssymb, amsthm,array,booktabs,graphicx,tikz,varwidth,float,enumerate}

\usetikzlibrary{shapes.geometric,arrows}
\tikzstyle{process} = [rectangle,minimum width=3cm, minimum height=1cm,text centered,draw=black,fill=white!10]
\tikzstyle{arrow} = [thick,->,>=stealth]

\usepackage{sectsty}
\allsectionsfont{\centering\normalsize}

\begin{document}

\begin{center}
	\textbf{\large Formula for the \boldmath$n$th \boldmath$k$-Generalized Fibonacci-like Number}
	\vspace{2mm}	
	\begin{center}
	\small	John Alexis B. Gemino$^*$\\ Alexander J. Balsomo\\ Geneveve M. Parre\~{n}o-Lachica\\ Dave Ryll B. Libre\\ Marc Raniel A. Nu\~{n}eza\\
		{\it Department of Mathematics, West Visayas State University, Iloilo City 5000, Philippines}\\
		$^*${\tt johnalexis.gemino@wvsu.edu.ph}
	\end{center}
\end{center}

\section*{Abstract}
 
\hspace{7mm} For a linear operator $T$ on $\mathbb{R}^2$ defined by $T(x,y)=(y,x+y),$ the Fibonacci sequence, $F_0=0$, $F_1=1$, and $F_n=F_{n-1}+F_{n-2}$ for $n \geq 2$, can be expressed as $T^n(0,1)=(F_n,F_{n+1})$. With eigenvalues $\lambda_i,\lambda_j$ of $T$ and corresponding eigenvectors $v_i,v_j$, the expression above becomes $T^n(0,1)=a_i \lambda_i^n v_i+ a_j \lambda_j^n v_j,$ and hence used to find $F_n$. In this work, we used the same system to find formula for determining the $n$th term of the Tribonacci and Tetrabonacci sequence, and for the $k$-generalized Fibonacci sequence with $k$ initial terms $0, 0,0,\ldots,1$. For arbitrary initial terms, formula for finding the $n$th Fibonacci-like, Tribonacci-like and Tetrabonacci-like number were also derived in the same fashion. In general, the $n$th $k$-generalized Fibonacci-like number, denoted as $\mathfrak{F}_n^{(k)}$, is given by  
		$$\mathfrak{F}_n^{(k)}=\sum_{m=1}^{k} \left[\displaystyle \frac{\displaystyle\sum_{p=1}^{k}\left(\lambda_m^{k-p}-\displaystyle\sum_{i=1}^{k-p}\lambda_m^{(k-p)-i}\right)t_{p-1}}{\displaystyle \prod_{j=1, \hspace{1mm} j\neq m}^{k}(\lambda_m - \lambda_j)} \right] \lambda_m^n, \hspace{2mm}(n \geq k \geq 2)$$
where $\lambda_1,\ldots,,\lambda_{k-1},\lambda_k$ are the roots of the characteristic polynomial $P(\lambda)=\lambda^k - \lambda^{k-1}-\cdots-\lambda-1$ and $t_0, t_1,\ldots, t_{k-1}$ are the $k$ arbitrary initial terms.

\vspace{5mm}

\textbf{Keywords:} $k$-generalized Fibonacci numbers, $k$-generalized Fibonacci-like numbers

\section*{Introduction}

A sequence is a list of numbers arranged in a specific order.  The numbers in a sequence, separated by commas, are called the terms of the sequence. Some sequences like of arithmetic, geometric, square, cube, Lucas, and Fibonacci are known to many for some time. In discrete mathematics, sequences can be used to represent solutions to certain counting problems. They are also an important data structure in computer science (Rosen, 2012). 

One of the most intriguing number sequences, the Fibonacci sequence, is a sequence starting from the initial terms, 0 and 1, where the succeeding terms are taken from adding the two previous terms (Koshy, 2007). 

The Fibonacci sequence appear in quite unexpected places. They occur in nature, music, geography, and geometry. They can be found in the spiral arrangements of seeds in sunflowers, the scale patterns of pine cones, the number of petals in flowers, and the arrangement of leaves on trees. It is so fascinating that there is an association of Fibonacci enthusiasts founded in 1963, \textit{The Fibonacci Association}, devoted to the study of the sequence. The association publishes \textit{The Fibonacci Quarterly} devoted to articles related to the Fibonacci sequence (Koshy, 2007).

Being one of the widely studied sequences, formula to derive the terms are well-known. One of which is the famous Binet's formula, an explicit formula for finding the $n$th term of the Fibonacci sequence, given by the equation
		\begin{equation}
			f_{n}=\displaystyle \frac{1}{\sqrt{5}}
			\left[\left(\displaystyle \frac{1+\sqrt{5}}{2}\right)^n+\left(\displaystyle \frac{1-\sqrt{5}}{2}\right)^n\right] \notag
		\end{equation}
where $f_n$ is the $n^{th}$ Fibonacci number with inital terms $f_0=0$ and $f_1=1$.

The Fibonacci number $f_{n}$ can also be derived by summing up the elements on the rising diagonal lines in the Pascal's triangle,
\begin{center}
	$1$\\
	$1 \hspace{5mm} 1$\\
	$1 \hspace{5mm} 2 \hspace{5mm} 1$\\
	$1 \hspace{5mm} 3 \hspace{5mm} 3 \hspace{5mm} 1$\\
	$1 \hspace{5mm} 4 \hspace{5mm} 6 \hspace{5mm} 4 \hspace{5mm} 1$\\
	$1 \hspace{5mm} 5 \hspace{5mm} 10 \hspace{5mm} 10 \hspace{5mm} 5 \hspace{5mm} 1.$\\
\end{center}

Many authors also have studied the various generalizations of the Fibonacci sequence. As early as 1975, Wong and Maddocks modified the Pascal's triangle, shown below is the first 6 rows,
\begin{center}
	$1$\\
	$1 \hspace{5mm} 1$\\
	$1 \hspace{5mm} 3 \hspace{5mm} 1$\\
	$1 \hspace{5mm} 5 \hspace{5mm} 5 \hspace{5mm} 1$\\
	$1 \hspace{5mm} 7 \hspace{5mm} 13 \hspace{5mm} 7 \hspace{5mm} 1$\\
	$1 \hspace{5mm} 9 \hspace{5mm} 25 \hspace{5mm} 25 \hspace{5mm} 9 \hspace{5mm} 1$\\
\end{center}
and revealed that sums of elements on the rising diagonal lines in their triangle gave the terms in the Tribonacci sequence, a variation of the Fibonacci sequence, with initial terms 0, 0 and 1, where the succeeding terms are taken from adding the three previous terms.

Natividad and Policarpio (2013), also presented a formula  in determining the $n$th term of the tribonacci-like sequence, a derivative of Tribonacci sequence, where the initial terms in the sequence are arbitrarily assigned. However such formula involves the determination of some tribonacci numbers in the first place. 

On the other hand, Meink (2011) approached the study of Fibonacci Numbers using the theory of diagonalizing a matrix and examining the eigenvalues of 2 x 2 generating matrices to derive the Binet type formula for the Fibonacci and its derivative, the Fibonacci - like sequence, a sequence where the initial terms are arbitrary numbers. Meinke (2011) also extended these results to tribonacci and tribonacci-like numbers with a similar $3\times 3$ matrix approach.


Furthermore, many authors also derived a formula for the $n$th term of the $k$-generalized Fibonacci sequence, a generalization for the Fibonacci sequence with $k$ initial terms, composed of $(k-1)$ 0's and 1, where the succeeding terms are generated by adding the previous $k$ terms. In 2017, Kuhapatanakul and Anantakitpaisal introduced the $k$-generalized Pascal's triangle in the form, 
	\begin{center}
		$C_k\binom{0}{0}$ \vspace{2mm}\\
		$C_k\binom{1}{0} \hspace{10mm} C_k\binom{1}{1}$\vspace{2mm}\\
		$C_k\binom{2}{0} \hspace{10mm} C_k\binom{2}{1} \hspace{10mm} C_k\binom{2}{2}$\vspace{2mm}\\
		$C_k\binom{3}{0} \hspace{10mm} C_k\binom{3}{1} \hspace{10mm} C_k\binom{3}{2} \hspace{10mm} C_k\binom{3}{3}$\vspace{2mm}\\
		$\vdots \hspace{30mm} \vdots$	
	\end{center}
with the definition: let $n, i$ be integers with $n\geq 0$, and define
	\begin{equation}
		C_k (n,i)=C_k(n-1,i)+ \sum_{j=1}^{k-1}C_k(n-j,i-1) \notag
	\end{equation}
where $C_k(n,0)=C_k(n,n)=1$, and $C_k(n,i)=0 $ for $ i<0 $ or $i>n$. They presented that the sums of elements on each rising diagonal line in the $k$-generalized Pascal's triangle gives the $k$-generalized Fibonacci number.

Dresden (2014) provided a Binet-style formula, a new representation, of the $n$th $k$-generalized Fibonacci number $f^{(k)}_n$ given by the equation
			\begin{equation}
				f^{(k)}_n = \displaystyle \sum_{i=1}^{k} \dfrac{\alpha_i - 1}{2+(k+1)(\alpha_i - 2)} \alpha_i^{n-1} \notag
			\end{equation}
where $\alpha_1, \ldots, \alpha_k$ are the roots of the polynomial function $P(x)=x^k-x^{k-1}-\ldots-1$.

Bacani and Rabago (2015) incorporated the above equation by Dresden (2014) in deriving the formula for the $n$th term, $G_n^{(k)}$, of the k-generalized Fibonacci-like sequence, a derivative of the k-generalized Fibonacci sequence having $k$ arbitrary initial terms. $G_n^{(k)}$ is given as follows: 
	\begin{equation}
		G_n^{(k)} = G_0^{(k)} \displaystyle \sum_{i=1}^{k}A(i;k)\alpha_i^{n-2}+ \displaystyle \sum_{m=0}^{k-3}\left( G_{m+1}^{(k)} \displaystyle \sum_{j=0}^{m+1} \displaystyle \sum_{i=1}^{k} A(i;k)\alpha_i^{n-2-j} \right) + G_{k-1}^{(k)} \displaystyle \sum_{i=1}^{k}A(i;k)\alpha_i^{n-1} \notag
	\end{equation}
where $A(i;k)=(\alpha_i-1)[2+(k+1)(\alpha_i - 2)]^{-1}$ and $\alpha_1, \alpha_2, \ldots, \alpha_k$ are roots of $P(x)=x^k-x^{k-1}- \cdots -1$ and $G_n^{(k)}$, for $n=0,1,\dots,k-1$, are the $k$ arbitrary initial terms.\vspace{2mm}
   
This work aims to contribute a formula for the $n$th $k$-generalized Fibonacci-like number with $k$-arbitrary initial terms, using matrices but without involving the diagonalization of a matrix.  

A square matrix $M$ of a linear operator $T$ on a complex vector space $V$, defined to generate a desired sequence, is diagonalizable if there exist an invertible matrix $S$, s.t.
$D=S^{-1}MS$ is a diagonal matrix. But such process, requires finding the determinant of $S$, and the classical adjoint of $S$ to deduce $S^{-1}$, not to consider the matrix multiplication that comes after. On the other hand, the basis for the formula of Bacani and Rabago (2015) for $k$-generalized Fibonacci-like numbers uses $k$-generalized Fibonacci numbers. 

The derivation of the formula for the $n$th $k$-generalized Fibonacci-like number was the generalization of finding the eigenvalues and eigenvectors of the operator
				\begin{equation}
					T^n(0,1)=(F_n,F_{n+1}), \notag
				\end{equation}
and to derive the Binet's formula for the nth Fibonacci number $(F_n)$ using the eigenvectors of $T$.  



\pagebreak 
\section*{Methods}

This work focuses on deriving a formula for the $n$th $k$-generalized Fibonacci-like number.  To come-up with a formula for the said number, we commence by defining linear operators $T_k$, for $k=2,3,4$ on $\mathbb{R}^2$, $\mathbb{C}^3$, and $\mathbb{C}^4$, associated with the Fibonacci, Tribonacci and Tetrabonacci sequence, respectively- the simpler cases of the goal of this paper. The matrix $M(T_k)$ corresponding to the operator $T_k$ is then identified. 

The initial terms for the mentioned sequences are used as the $k$-tupled vector $(0,0,\ldots,1)$ to show that
		\begin{equation}
			T^n_k (0,0, \ldots, 1) = (F^k_n, F^k_{n+1}, \ldots, F^k_{n+(k-1)}).
		\end{equation}
where $F^k_n$, for $k=2,3,4$, is the nth term of the Fibonacci, Tribonacci and Tetrabonacci sequence, respectively. 

Also, the eigenvalues and corresponding eigenvectors for each $T_k$, for $k=2,3,4$ are determined. The vector form of the initial terms of the Fibonacci, Tribonacci and Tetrabonacci sequence, are then written as linear combinations  using the eigenvectors as basis vectors. The following equation will be shown,
		\begin{equation}
			T^n_k (0,0, \ldots, 1) = a_1 \lambda_1^n v_1 + a_2 \lambda_2^n v_2 + \ldots + a_k \lambda_k^n v_k.
		\end{equation}
where the $a_k$'s are scalar values.

The values of $a_k$s are identified by evaluating (2) at $n=0$. With (1) and (2), we deduce that
	\begin{equation}
		(F^k_n, F^k_{n+1}, \ldots, F^k_{n+(k-1)}) =  a_1 \lambda_1^n v_1 + a_2 \lambda_2^n v_2 + \ldots + a_k \lambda_k^n v_k,
	\end{equation}
 for $k=2,3,4$. 
 
 A formula for the $n$th term of the Fibonacci, Tribonacci, and Tetrabonacci sequence are found using the first vector components of (3), for $k=2,3,4$, respectively. Patterns are then recognized and used to derive the formula for the $n$th term of $k$-generalized Fibonacci sequence.  

Using arbitrary initial terms, $t_0, t_1, \ldots,t_k-1$, as the $k$-tupled vector $(t_0, t_1, ..., t_{k-2},t_k-1)$ in (1) and (2), and applying the same process, we derive the nth term of the Fibonacci-like, Tribonacci-like, and Tetrabonacci sequence, denoted by $\mathfrak{F}^k_n$, for $k=2,3,4$, respectively. Patterns are also recognized and used to derive the formula for the $n$th term of $k$-generalized Fibonacci-like sequence.

\pagebreak
\section*{Results and Discussion}

The following are the formulas derived for the $n$th term of Fibonacci, Tribonacci and Tetrabonacci sequence.

\begin{itemize}
	\item  The $n$th Fibonacci number with initials terms, $F_0^{(2)}=0,F_1^{(2)}=1$, is given by
				\begin{equation}
					F_n^{(2)}=	\displaystyle \frac{\lambda_1^n}{\lambda_1 - \lambda_2}+ \displaystyle \frac{\lambda_2^n}{\lambda_2 - \lambda_1}, 
				\end{equation}
			where $\lambda_1,\lambda_2$ are the roots of the polynomial $P_1(\lambda)=\lambda^2-\lambda-1$.
			
	\item The $n$th Tribonacci number with initial terms, $F_0^{(3)}=F_1^{(3)}=0, F_{2}^{(3)}=1$, is given by
				\begin{equation}
					F_n^{(3)}=
					\displaystyle \frac{\lambda_1^n}{(\lambda_1 - \lambda_2)(\lambda_1 - \lambda_3)}+
					\displaystyle \frac{\lambda_2^n}{(\lambda_2 - \lambda_1)(\lambda_2 - \lambda_3)}+	\displaystyle \frac{\lambda_3^n}{(\lambda_3 - \lambda_1)(\lambda_3 - \lambda_2)}
				\end{equation}
			where $\lambda_1,\lambda_2,\lambda_3$ are the roots of the polynomial $P_2(\lambda)=\lambda^3-\lambda^2-\lambda-1$.
	
	\item The $n$th Tetrabonacci number with initial terms,	$F_0^{(4)}=F_1^{(4)}=F_2^{(4)}=0, F_3^{(4)}=1$, is given by	
					\begin{eqnarray}
						F_n^{(4)}&=&
						\displaystyle \frac{\lambda_1^n}{(\lambda_1 - \lambda_2)(\lambda_1 - 		
							\lambda_3)(\lambda_1 - \lambda_4)}+
						\displaystyle \frac{\lambda_2^n}{(\lambda_2 - \lambda_1)(\lambda_2 - 
							\lambda_3)(\lambda_2 - \lambda_4)}\notag\\
						&+&
						\displaystyle \frac{\lambda_3^n}{(\lambda_1 - \lambda_2)(\lambda_1 - 
							\lambda_3)(\lambda_1 - \lambda_3)}+
						\displaystyle \frac{\lambda_4^n}{(\lambda_1 - \lambda_2)(\lambda_1 - 
							\lambda_3)(\lambda_1 - \lambda_3)},
					\end{eqnarray}
		where $\lambda_1,\lambda_2,\lambda_3,\lambda_4$ are the roots of the polynomial $P_3(\lambda)=\lambda^4-\lambda^3-\lambda^2-\lambda-1$.
\end{itemize}

Each of the formula above were rewritten in summation notation and summarized in the table below to find a recognizable pattern.

{\footnotesize $$\begin{array}{c c c}
		\hline \\ 
	\textbf{Name of Sequence} 			&\textbf{Initial Terms}		&  \textbf{nth Term of the Sequence} \\[7pt] 
	\text{Fibonacci Sequence}			&	0,1						&  F_n^{(2)}=\displaystyle \sum_{m=1}^{2} \frac{\lambda_m^n}{\displaystyle\prod_{j=1,\hspace{1mm} j\neq m}^{2}(\lambda_m-\lambda_j)}, \hspace{2mm} (n\geq 2)	 \\
	\text{Tribonacci Sequence}			&	0,0,1					& F_n^{(3)}=\displaystyle \sum_{m=1}^{3} \frac{\lambda_m^n}{\displaystyle\prod_{j=1,\hspace{1mm} j\neq m}^{3}(\lambda_m-\lambda_j)}, \hspace{2mm} (n\geq 3)	\\
	\text{Tetrabonacci Sequence}		&	0,0,0,1					&  F_n^{(4)}=\displaystyle \sum_{m=1}^{4} \frac{\lambda_m^n}{\displaystyle\prod_{j=1,\hspace{1mm} j\neq m}^{4}(\lambda_m-\lambda_j)}, \hspace{2mm} (n\geq 4) \\
			\vdots						& 	\vdots 					& 	\vdots  \\
	\text{$k$-generalized Fibonacci Sequence} & 0, 0, \ldots,0, 1  &	F_n^{(k)}=\displaystyle \sum_{m=1}^{k} \frac{\lambda_m^n}{\displaystyle\prod_{j=1,\hspace{1mm} j\neq m}^{k}(\lambda_m-\lambda_j)}, \hspace{2mm} (n\geq k \geq 2) \\[35pt] \hline 
\end{array}$$}

By observing how the pattern goes, we propose the following theorem.

\newtheorem{theorem}{Theorem}
\begin{theorem}
The $nth$ $k$-generalized Fibonacci number, denoted by $F_n^{(k)}$, with $k$ initial terms, $F_0^{(k)}=F_1^{(k)}=...=F_{k-2}^{(k)}=0$ and $F_{k-1}^{(k)}=1$, is
		\begin{equation}
			F_n^{(k)}=\sum_{m=1}^{k}\displaystyle\frac{\lambda_m^n}{\displaystyle\prod_{j=1,\hspace{1mm} j\neq m}^{k}(\lambda_m-\lambda_j)}, \hspace{2mm} (n\geq k \geq 2) 
		\end{equation}
where $\lambda_1, \lambda_2,\ldots,\lambda_{k-1},\lambda_k$ are the roots of the characteristic polynomial $P(\lambda)=\lambda^k - \lambda^{k-1}- \ldots -\lambda-1$.	
\end{theorem}

\begin{proof}
We need to show that our formula and Dresden's $(2014)$ presentation of the $nth$ $k$-generalized Fibonacci number, 
$$\mathbf{F_n^{(k)}}= \sum_{m=1}^{k} \displaystyle \frac{\lambda_m - 1}{2+(k+1)(\lambda_m - 2)}\lambda_m^{n-1},$$
with initial conditions $F_{-(k-2)}=F_{-(k-1)}=\ldots=F_0=0$ and $F_1=1$ are the same, that is, we need to show the equality
		\begin{equation}
			\mathbf{F_n^{(k)}}=F_{n+k-2}^{(k)}.
		\end{equation}
So,
	\begin{equation}
		\sum_{m=1}^{k} \displaystyle \frac{\lambda_m - 1}{2+(k+1)(\lambda_m - 2k)}\lambda_m^{n-1}
		\newcommand{\?}{\stackrel{?}{=}} \?
		\sum_{m=1}^{k}\displaystyle\frac{\lambda_m^{n+k-2}}{\displaystyle\prod_{j=1,\hspace{1mm} j\neq m}^{k}(\lambda_m-\lambda_j)}
	\end{equation}
Manipulating the left hand side (LHS) of the equation above,
	\begin{equation}
		\sum_{m=1}^{k} \displaystyle \frac{\lambda_m - 1}{[(k+1)\lambda_m - 2)]\lambda_m^{k-1}}(\lambda_m^{n+k-2})
	 	 \newcommand{\?}{\stackrel{?}{=}} \? \sum_{m=1}^{k}\displaystyle\frac{\lambda_m^{n+k-2}}{\displaystyle\prod_{j=1,\hspace{1mm} j\neq m}^{k}(\lambda_m-\lambda_j)}. \notag
	\end{equation}
Thus, for (8) to be true, we need to show the following product identity
	\begin{equation}
		\frac{[(k+1)\lambda_m - 2k]\lambda_m^{k-1}}{\lambda_m - 1} = \prod_{j=1, \hspace{2mm} j \neq m}^{k}(\lambda_m - \lambda_j).
	\end{equation}
Without loss of generality, let $m=1$, and let the right hand side (RHS) of (10) be $A$, so
		\begin{equation}
			A = \prod_{j=2, \hspace{2mm} j \neq 1}^{k}(\lambda_1 - \lambda_j) = (\lambda_1 - \lambda_2)(\lambda_1 - \lambda_3) \ldots (\lambda_1 - \lambda_{k-1})(\lambda_1 - \lambda_k)
		\end{equation}
By expansion of the linear factorization of the RHS of $A$,
\begin{eqnarray}
	A &=&  \lambda_1^{k-1}-e_1(\lambda_2,\lambda_3,\ldots,\lambda_k)\lambda_1^{k-2}+e_2(\lambda_2,\lambda_3,\ldots,\lambda_k)\lambda_1^{k-3}	\notag\\
		&& -e_3(\lambda_2,\lambda_3,\ldots,\lambda_k)\lambda_1^{k-4} + \ldots + (-1)^{k-1} e_{k-1}(\lambda_2,\lambda_3,\ldots,\lambda_k)
\end{eqnarray}
where, 
		\begin{eqnarray}
			e_1(\lambda_2,\lambda_3,\ldots,\lambda_k) &=& \lambda_2 + \lambda_3 + \ldots + \lambda_{k-1} + \lambda_k \notag\\
			e_2(\lambda_2,\lambda_3,\ldots,\lambda_k) &=& \lambda_2\lambda_3 + \lambda_2\lambda_4 +\ldots\lambda_2\lambda_k+\lambda_3\lambda_4+\ldots+\lambda_3\lambda_k + \ldots + \lambda_{k-1} + \lambda_k \notag\\
			e_3(\lambda_2,\lambda_3,\ldots,\lambda_k) &=& \lambda_2\lambda_3\lambda_4 + \lambda_2\lambda_3\lambda_5 +\ldots + \lambda_{k-2}\lambda_{k-1}\lambda_k \notag\\
			\vdots \hspace{5mm}&\vdots& \hspace{5mm}\vdots \notag\\
			e_{k-2}(\lambda_2,\lambda_3,\ldots,\lambda_k) &=& \lambda_2\lambda_3\ldots\lambda_{k-1}+\lambda_2\lambda_3\ldots\lambda_{k-2}\lambda_k
			+\ldots + \lambda_2\lambda_4\ldots\lambda_{k-1}\lambda_k+\lambda_3\lambda_4\ldots\lambda_{k-1}\lambda_k \notag\\
			e_{k-1}(\lambda_2,\lambda_3,\ldots,\lambda_k) &=& \lambda_2\lambda_3\ldots\lambda_k \notag
		\end{eqnarray}
Multiplying both sides of $(12)$ by $\lambda_1$, we have
	\begin{eqnarray}
		\lambda_1 A &=&  		
		\lambda_1^k-e_1(\lambda_2,\lambda_3,\ldots,\lambda_k)\lambda_1^{k-1}+
		e_2(\lambda_2,\lambda_3,\ldots,\lambda_k)\lambda_1^{k-2} \\
		&&+\ldots+
		(-1)^{k-1} e_{k-1}(\lambda_2,\lambda_3,\ldots,\lambda_k)\lambda_1. \notag
	\end{eqnarray}
Subtracting $(12)$ from $(13)$, 
	\begin{eqnarray}
		(\lambda_1 -1 )A &=&  		
		\lambda_1^k-\lambda_1^{k-1} -  e_1(\lambda_2,\lambda_3,\ldots,\lambda_k)(\lambda_1^{k-1}-\lambda_1^{k-2}) + e_2(\lambda_2,\lambda_3,\ldots,\lambda_k)(\lambda_1^{k-2}-\lambda_1^{k-3}) \notag\\
		&& +\ldots+
		(-1)^{k-1} e_{k-1}(\lambda_2,\lambda_3,\ldots,\lambda_k)(\lambda_1-1)
	\end{eqnarray}
Multiplying both sides by $\lambda^2_1$ we have,
		\begin{eqnarray}
			\lambda_1^2(\lambda_1 -1 )A &=& \lambda_1^{k+2}-\lambda_1^{k+1} -  e_1(\lambda_2,\lambda_3,\ldots,\lambda_k)(\lambda_1^{k+1}-\lambda_1^k)+
			e_2(\lambda_2,\lambda_3,\ldots,\lambda_k)(\lambda_1^{k}-\lambda_1^{k-1}) \notag\\
			&& +\ldots+
			(-1)^{k-1} e_{k-1}(\lambda_2,\lambda_3,\ldots,\lambda_k)(\lambda_1^3-\lambda_1^2) \notag\\
			&=&\lambda_1^k [\lambda_1^2-\lambda_1 -  e_1(\lambda_2,\lambda_3,\ldots,\lambda_k)(\lambda_1 - 1)+
			e_2(\lambda_2,\lambda_3,\ldots,\lambda_k)(1-\lambda_1^{-1}) \notag\\
			&& +\ldots+ (-1)^{k-1} e_{k-1}(\lambda_2,\lambda_3,\ldots,\lambda_k)(\lambda_1^{-(k-3)}-\lambda_1^{-(k-2)})].
		\end{eqnarray}
Let the expression enclosed with brackets on the RHS of $(15)$ be $A_1$, so 
\begin{eqnarray}
	A_1 &=& \lambda_1^2-\lambda_1 -  e_1(\lambda_2,\lambda_3,\ldots,\lambda_k)(\lambda_1 - 1)+
	e_2(\lambda_2,\lambda_3,\ldots,\lambda_k)(1-\lambda_1^{-1}) \notag\\
	&& +\ldots+ (-1)^{k-1} e_{k-1}(\lambda_2,\lambda_3,\ldots,\lambda_k)(\lambda_1^{k-3}-\lambda_1^{k-2}) \notag\\
	&=& (\lambda_1^2-\lambda_1)-(1 - \lambda_1)(\lambda_1 - 1) + (-1 - [\lambda_1 \lambda_2 + \lambda_1 \lambda_3 + \ldots + \lambda_1 \lambda_k])(1-\lambda_1^{-1})\notag\\
	&& + \ldots + (-1)^{k-1}(-1)^k \frac{(-1)}{\lambda_1}(\lambda_1^{-(k-3)}-\lambda_1^{-(k-2)}) \notag\\
	&=& (\lambda_1^2-\lambda_1)+(\lambda_1 - 1)(\lambda_1 - 1) + (\lambda_1^2 - \lambda_1 - 1)(1-\lambda_1^{-1})\notag\\
	&& + \ldots + \frac{1}{\lambda_1}(\lambda_1^{-(k-3)}-\lambda_1^{-(k-2)}). \notag
\end{eqnarray}
$A_1$ can also be be expressed as follows:
\begin{eqnarray}
	\lambda_1 (\lambda_1 - 1) &+& \notag\\
	(\lambda_1 - 1)(\lambda_1 - 1)  &+& \notag\\
	\left(\lambda_1 - 1 - \frac{1}{\lambda_1}\right)(\lambda_1 - 1)	&+& \notag\\
	\vdots \hspace{14mm} && \notag\\
	\left( \lambda_1 - 1 - \frac{1}{\lambda_1} - \ldots  - \frac{1}{\lambda_1^{k-2}} \right)(\lambda_1 - 1). && \notag
\end{eqnarray}
Factoring $(\lambda_1 - 1)$ out, we have 
\begin{eqnarray}
	A_1 &=& (\lambda_1 - 1) \left[ \lambda_1 + (\lambda_1 - 1) + \left(\lambda_1 - 1 - \frac{1}{\lambda_1}\right) + \ldots +  \left(\lambda_1 - 1 - \frac{1}{\lambda_1} - \ldots  - \frac{1}{\lambda_1^{k-2}}\right) \right] \notag\\
	A_1 &=& (\lambda_1 - 1) \left[ k\lambda_1 - (k-1) - (k-2)\frac{1}{\lambda_1} - \ldots - \frac{2}{\lambda_1^{k-3}} - \frac{1}{\lambda_1^{k-2}}\right].
\end{eqnarray}
Now, let $A_2$ the expression enclosed with brackets on the RHS of $A_1$, hence, 
\begin{eqnarray}
	A_2 &=&  k\lambda_1 - (k-1) - (k-2)\frac{1}{\lambda_1} - \ldots - \frac{2}{\lambda_1^{k-3}} - \frac{1}{\lambda_1^{k-2}} \\
	\lambda_1 A_2 &=& k^2 \lambda_1 - (k-1)\lambda_1 - (k-2 ) - \ldots - \frac{2}{\lambda_1^{k-4}} - \frac{1}{\lambda_1^{k-3}}.
\end{eqnarray}
Subtracting $(17)$ from $(18)$,
\begin{eqnarray}
	(\lambda_1 - 1) A_2 &=& k^2 \lambda_1 - (k-1)\lambda_1 - k \lambda_1 + \left( 1 + \frac{1}{\lambda_1} + \frac{1}{\lambda_1^2}+\ldots+\frac{1}{\lambda_1^{k-3}} \right) + \frac{1}{\lambda_1^{k-2}}  \notag\\
	&=& k \lambda_1^2 - k \lambda_1 + \lambda_1^2 - k \lambda_1 + 1 + \frac{1}{\lambda_1} + \frac{1}{\lambda_1^2}+\ldots+\frac{1}{\lambda_1^{k-3}} + \frac{1}{\lambda_1^{k-2}} \notag\\ 
	&=& k \lambda_1^2 - 2k \lambda_1 + \left[ \lambda_1 + 1 + \frac{1}{\lambda_1} + \frac{1}{\lambda_1^2}+\ldots+\frac{1}{\lambda_1^{k-3}}+ \frac{1}{\lambda_1^{k-2}} \right]\notag\\
	&=& k \lambda_1^2 - 2k \lambda_1 + \frac{\lambda_1 \left[1 - \left( \frac{1}{\lambda_1}\right)^k\right] }{1 - \frac{1}{\lambda_1}} \notag\\
	&=& k \lambda_1^2 - 2k \lambda_1 + \lambda_1 \left(1 - \frac{1}{\lambda_1^k} \right)\left( \frac{\lambda_1}{\lambda_1 - 1} \right) \notag\\
	&=& k \lambda_1^2 - 2k \lambda_1 + \left( \frac{\lambda_1^2}{\lambda_1 - 1} \right)\left( \frac{\lambda_1^k - 1}{\lambda_1^k} \right) \notag\\
	&=& k \lambda_1^2 - 2k \lambda_1 + \left( \frac{\lambda_1^2}{\lambda_1 - 1} \right)(\lambda_1^k - 1)[-(\lambda_1 - 2)] \notag\\
	&=& k \lambda_1^2 - 2k \lambda_1 + \frac{(\lambda_1^2)(-\lambda_1^{k+1}+ 2\lambda_1^k + \lambda_1 - 2)}{\lambda_1 - 1} \notag\\
	&=& k \lambda_1^2 - 2k \lambda_1 + \frac{(-\lambda_1^2)(\lambda_1^{k-1}- 2\lambda_1^k + 1 + 1 - \lambda_1)}{\lambda_1 - 1} \notag\\
	&=& k \lambda_1^2 - 2k \lambda_1 + (-\lambda_1^2)\left[ \frac{\lambda_1^{k+1}-2k\lambda_1^k + 1}{\lambda_1 - 1} +
	\frac{1 - \lambda_1}{\lambda_1 - 1}\right] \notag\\
	&=& k \lambda_1^2 - 2k \lambda_1 + (-\lambda_1^2)\left[ 0 -
	\frac{\lambda_1 - 1}{\lambda_1 - 1}\right] \notag\\
	&=& k \lambda_1^2 - 2k \lambda_1 + \lambda_1^2 \notag\\
	A_2 &=& \frac{(k+1)\lambda_1^2 - 2k \lambda_1}{\lambda_1 - 1}. \notag
\end{eqnarray}
Hence,
\begin{eqnarray}
	A_1 &=& (\lambda_1 - 1)A_2 \notag\\
	 &=& (\lambda_1 - 1) \frac{(k+1)\lambda_1^2 - 2k \lambda_1}{\lambda_1 - 1} \notag\\
	 &=& (k+1)\lambda_1^2 - 2k \lambda_1.
\end{eqnarray}
And,
\begin{eqnarray}
	\lambda_1^2(\lambda_1 -1 )A &=& \lambda^k A_1 \notag\\
	A &=& \frac{\lambda_1^k[(k+1)\lambda_1^2 - 2k \lambda_1]}{\lambda_1^2(\lambda_1 -1 )} \notag\\
	A &=& \frac{[(k+1)\lambda_1 - 2k]\lambda_1^{k-1}}{\lambda_1 -1}
\end{eqnarray}
Hence, for any $m$, $1 \leq m \leq k$,  
$$\frac{[(k+1)\lambda_m - 2k]\lambda_m^{k-1}}{\lambda_m - 1} = \prod_{j=1, \hspace{2mm} j \neq m}^{k}(\lambda_m - \lambda_j)$$
Therefore, our formula is the same with Dresden's formula for the $nth$ $k$-generalized Fibonacci number.
\end{proof}


The following are the formulas derived for the nth term of Fibonacci-like, Tribonacci-like and Tetrabonacci-like sequence.

\begin{itemize}
	\item The $nth$ Fibonacci-like number with initial terms $t_0$, $t_1$ 
				\begin{equation}
					\mathfrak{F}_n^{(2)}=
					\left[\displaystyle \frac{(\lambda_1 - 1)t_0 + t_1}
					{\lambda_1 - \lambda_2}\right]\lambda_1^n +
					\left[\displaystyle \frac{(\lambda_2 - 1)t_0 + t_1}
					{\lambda_2 - \lambda_1}\right]\lambda_2^n
				\end{equation}
		where $\lambda_1,\lambda_2$ are the roots of the polynomial $P_3(\lambda)=\lambda^2-\lambda-1$.
	\item The $nth$ Tribonacci-like number with initial terms $t_0$, $t_1$, $t_2$ is given by 
				\begin{eqnarray}
					\mathfrak{F}_n^{(3)}&=&
					\left[\displaystyle \frac{(\lambda_1^2 - \lambda_1 - 1)t_0 + (\lambda_1 - 1)t_1 + t_2}{(\lambda_1 - \lambda_2)(\lambda_1 - \lambda_3)}\right]\lambda_1^n \notag\\
					&+&
					\left[\displaystyle \frac{(\lambda_2^2 - \lambda_2 - 1)t_0 + (\lambda_2 - 1)t_1 + t_2}{(\lambda_2 - \lambda_1)(\lambda_2 - \lambda_3)}\right]\lambda_2^n \notag\\
					&+&
					\left[\displaystyle \frac{(\lambda_3^2 - \lambda_3 - 1)t_0 + (\lambda_3 - 1)t_1 + t_2}{(\lambda_3 - \lambda_1)(\lambda_3 - \lambda_2)}\right]\lambda_3^n \notag
				\end{eqnarray}
			where $\lambda_1,\lambda_2,\lambda_3$ are the roots of the polynomial $P_3(\lambda)=\lambda^3-\lambda^2-\lambda-1$.
	\item The $nth$ Tetrabonacci-like number with initial terms $t_0$, $t_1$, $t_2$, $t_3$
				\begin{eqnarray}
					\mathfrak{F}_n^{(4)}&=&
					\left[\displaystyle \frac{(\lambda_1^3 - \lambda_1^2 - \lambda_1 - 1)t_0 + (\lambda_1^2 - \lambda_1 - 1)t_1 + (\lambda_1 - 1)t_2 + t_3}{(\lambda_1 - \lambda_2)(\lambda_1 - \lambda_3)(\lambda_1 - \lambda_4)}\right]\lambda_1^n \notag\\
					&+&
					\left[\displaystyle \frac{(\lambda_2^3 - \lambda_2^2 - \lambda_2 - 1)t_0 + (\lambda_2^2 - \lambda_2 - 1)t_1 + (\lambda_2 - 1)t_2 + t_3}{(\lambda_2 - \lambda_1)(\lambda_2 - \lambda_3)(\lambda_2 - \lambda_4)}\right]\lambda_2^n \notag\\
					&+&
					\left[\displaystyle \frac{(\lambda_3^3 - \lambda_3^2 - \lambda_3 - 1)t_0 + (\lambda_3^2 - \lambda_3 - 1)t_1 + (\lambda_3 - 1)t_2 + t_3}{(\lambda_3 - \lambda_1)(\lambda_3 - \lambda_2)(\lambda_3 - \lambda_4)}\right]\lambda_3^n \notag\\
					&+&
					\left[\displaystyle \frac{(\lambda_4^3 - \lambda_4^2 - \lambda_4 - 1)t_0 + (\lambda_4^2 - \lambda_4 - 1)t_1 + (\lambda_4 - 1)t_2 + t_3}{(\lambda_4 - \lambda_1)(\lambda_4 - \lambda_2)(\lambda_4 - \lambda_3)}\right]\lambda_4^n \notag
				\end{eqnarray}
			where $\lambda_1,\lambda_2,\lambda_3,\lambda_4$ are the roots of the polynomial $P_3(\lambda)=\lambda^4-\lambda^3-\lambda^2-\lambda-1$.
\end{itemize}

Each of the formula above were rewritten in summation notation and summarized in the table below to find a recognizable pattern.

{\small $$\begin{array}{ccc}
			\hline\\
		\textbf{Name of Sequence} 			&\textbf{Initial Terms}		&  \textbf{nth Term of the Sequence} \\[7pt]  
		\text{Fibonacci-like}			&	t_0,t_1				&  \mathfrak{F}_n^{(2)}=\displaystyle \sum_{m=1}^{2} \left[ \frac{\displaystyle\sum_{p=1}^{2}\left(\lambda_m^{2-p}-\displaystyle\sum_{i=1}^{2-p}\lambda_m^{(2-p)-i}\right)t_{p-1}}{\displaystyle \prod_{j=1, j\neq m}^{2}(\lambda_m - \lambda_j)} \right] \lambda_m^n \\
		\text{Tribonacci-like}			&	t_0,t_1,t_2		&  \mathfrak{F}_n^{(3)}=\displaystyle \sum_{m=1}^{3} \left[ \frac{\displaystyle\sum_{p=1}^{3}\left(\lambda_m^{3-p}-\displaystyle\sum_{i=1}^{3-p}\lambda_m^{(3-p)-i}\right)t_{p-1}}{\displaystyle \prod_{j=1, j\neq m}^{3}(\lambda_m - \lambda_j)} \right] \lambda_m^n \\
		\text{Tetrabonacci-like}		&	t_0,t_1,t_2,t_3		&  \mathfrak{F}_n^{(4)}=\displaystyle \sum_{m=1}^{4} \left[ \frac{\displaystyle\sum_{p=1}^{4}\left(\lambda_m^{4-p}-\displaystyle\sum_{i=1}^{4-p}\lambda_m^{(4-p)-i}\right)t_{p-1}}{\displaystyle \prod_{j=1, j\neq m}^{4}(\lambda_m - \lambda_j)} \right] \lambda_m^n \\
		\vdots						& 	\vdots 					& 	\vdots  \\
		\text{$k$-generalized Fibonacci-like} & t_0, t_1, \ldots, t_{k-1}  & \mathfrak{F}_n^{(k)}=\displaystyle\sum_{m=1}^{k} \left[ \frac{\displaystyle\sum_{p=1}^{k}\left(\lambda_m^{k-p}-\displaystyle\sum_{i=1}^{k-p}\lambda_m^{(k-p)-i}\right)t_{p-1}}{\displaystyle \prod_{j=1, j\neq m}^{k}(\lambda_m - \lambda_j)} \right] \lambda_m^n \\[35pt] \hline
	\end{array}$$}

By observing how the pattern goes, we propose the following theorem.

\begin{theorem}
The $nth$ $k$-generalized Fibonacci-like number, denoted by $(\mathfrak{F}_n^{(k)})$, with $k$ arbitrary initial terms, $t_0$, $t_1$,$\ldots$, $t_{k-2}$, $t_{k-1}$, is

\begin{equation}
	\mathfrak{F}_n^{(k)}=\sum_{m=1}^{k} \left[\displaystyle \frac{\displaystyle\sum_{p=1}^{k}\left(\lambda_m^{k-p}-\displaystyle\sum_{i=1}^{k-p}\lambda_m^{(k-p)-i}\right)t_{p-1}}{\displaystyle \prod_{j=1, j\neq m}^{k}(\lambda_m - \lambda_j)} \right] \lambda_m^n, \hspace{2mm}(n \geq k \geq 2)
\end{equation}
where $\lambda_1, \lambda_2,...,\lambda_k$ are the roots of the characteristic polynomial $P(\lambda)=\lambda^k - \lambda^{k-1}- ...-\lambda-1$. 	
\end{theorem}

\begin{proof}
We will use the 2nd Principle of Mathematical Induction (Strong Induction) on $n$. Let $k$ be fixed.

\begin{enumerate}[(a)]
	\item \textbf{Base Case:}
	
		For $n=k$, WNTS that
							\begin{equation}
								\mathfrak{F}_k^{(k)}=t_0+t_1+t_2+ \ldots +t_{k-2}+t_{k-1}. 
							\end{equation}
		So, 
		\begin{equation}
			\mathfrak{F}_k^{(k)}=\sum_{m=1}^{k} \left[\displaystyle \frac{\displaystyle\sum_{p=1}^{k}\left(\lambda_m^{k-p}-\displaystyle\sum_{i=1}^{k-p}\lambda_m^{(k-p)-i}\right)t_{p-1}}{\displaystyle \prod_{j=1, j\neq m}^{k}(\lambda_m - \lambda_j)} \right] \lambda_m^k \notag
		\end{equation}

		\begin{eqnarray}
			\mathfrak{F}_k^{(k)}&=&\left[\displaystyle \frac{(\lambda_1^{k-1}- \ldots -1)t_0 + (\lambda_1^{k-2} - \ldots - 1)t_1 + \ldots +(\lambda_1-1)t_{k-2} + t_{k-1}}
			{(\lambda_1-\lambda_2)(\lambda_1-\lambda_3) \ldots (\lambda_1-\lambda_k)}\right] \lambda_1^k +
			\notag\\
			&& \left[\displaystyle \frac{(\lambda_2^{k-1} - \ldots - 1)t_0 + (\lambda_2^{k-2}- \ldots -1)t_1 + \ldots +(\lambda_2-1)t_{k-2} + t_{k-1}}
			{(\lambda_2-\lambda_1)(\lambda_2-\lambda_3) \ldots (\lambda_2-\lambda_k)}\right] \lambda_2^k +\ldots+
			\notag\\
			&&\left[\displaystyle \frac{(\lambda_{k-1}^{k-1}- \ldots -1)t_0 + (\lambda_{k-1}^{k-2}- \ldots -1)t_1 + \ldots +(\lambda_{k-1}-1)t_{k-2} + t_{k-1}}
			{(\lambda_{k-1}-\lambda_1)(\lambda_{k-1}-\lambda_2) \ldots (\lambda_{k-1}-\lambda_k)}\right] \lambda_{k-1}^k +
			\notag\\
			&&\left[\displaystyle \frac{(\lambda_k^{k-1}- \ldots -1)t_0 + (\lambda_k^{k-2}- \ldots -1)t_1 + \ldots +(\lambda_k-1)t_{k-2} + t_{k-1}}
			{(\lambda_k-\lambda_1)(\lambda_k-\lambda_2) \ldots (\lambda_k-\lambda_{k-1})}\right]\lambda_k^k
			\notag
		\end{eqnarray}
		Next, we need to regroup the RHS in terms of $t_0, t_1,\ldots, t_{k-2}, t_k-1$ and show that the coefficients of $t_0, t_1, \ldots, t_{k-2}, t_k-1$ are all equal to $1$.

		Solving for the coefficient of $t_{k-1}$:
			\begin{eqnarray}
				t_{k-1}&:& 
				\displaystyle\frac{\lambda_1^k}{(\lambda_1-\lambda_2)(\lambda_1-\lambda_3) \ldots (\lambda_1-\lambda_k)}+
				\displaystyle\frac{\lambda_2^k}{(\lambda_2-\lambda_1)(\lambda_2-\lambda_3) \ldots (\lambda_2-\lambda_k)}\notag\\
				&& +\ldots+ \displaystyle\frac{\lambda_{k-1}^k}{(\lambda_{k-1}-\lambda_1)(\lambda_{k-1}-\lambda_2) \ldots (\lambda_{k-1}-\lambda_k)}\notag\\
				&&+\displaystyle\frac{\lambda_k^k}{(\lambda_k-\lambda_1)(\lambda_k-\lambda_2) \ldots (\lambda_k-\lambda_{k-1})}
				\notag
			\end{eqnarray}
		But from (7), evaluating at $n=k$, we deduce that the coefficient of $t_{k-1}$ is $F_k^{(k)}$.
		
		For the coefficient of $t_{k-2}$, we have
		\begin{eqnarray}
			t_{k-2}	&:& \frac{(\lambda_1-1)\lambda_1^k}{(\lambda_1-\lambda_2)(\lambda_1-\lambda_3) \ldots (\lambda_1-\lambda_k)}
						+\frac{(\lambda_2-1)\lambda_2^k}{(\lambda_2-\lambda_1)(\lambda_2-\lambda_3) \ldots (\lambda_2-\lambda_k)} +\ldots+ \notag\\
					& &	\frac{(\lambda_{k-1}-1)\lambda_{k-1}^k}{(\lambda_{k-1}-\lambda_1)(\lambda_{k-1}-\lambda_2) \ldots (\lambda_{k-1}-\lambda_k)}
						+\frac{(\lambda_k-1)\lambda_k^k}{(\lambda_k-\lambda_1)(\lambda_k-\lambda_2) \ldots (\lambda_k-\lambda_{k-1})}\notag\\
					&=&	\frac{\lambda_1^{k+1}}{(\lambda_1-\lambda_2)(\lambda_1-\lambda_3) \ldots (\lambda_1-\lambda_k)}+
						\frac{\lambda_2^{k+1}}{(\lambda_2-\lambda_1)(\lambda_2-\lambda_3) \ldots (\lambda_2-\lambda_k)} +\ldots+\notag\\
					& &	\frac{\lambda_{k-1}^{k+1}}{(\lambda_{k-1}-\lambda_1)(\lambda_{k-1}-\lambda_2) \ldots (\lambda_{k-1}-\lambda_k)} 				
						+\frac{\lambda_k^{k+1}}{(\lambda_k-\lambda_1)(\lambda_k-\lambda_2) \ldots 	(\lambda_k-\lambda_{k-1})}\notag\\
					& &	-\left[\frac{\lambda_1^k}{(\lambda_1-\lambda_2)(\lambda_1-\lambda_3) \ldots (\lambda_1-\lambda_k)} 
						+\frac{\lambda_2^k}{(\lambda_2-\lambda_1)(\lambda_2-\lambda_3) \ldots (\lambda_2-\lambda_k)}+\ldots+ \right.\notag\\
					& &	\frac{\lambda_{k-1}^k}{(\lambda_{k-1}-\lambda_1)(\lambda_{k-1}-\lambda_2) \ldots (\lambda_{k-1}-\lambda_k)} \left. + 	
						\frac{\lambda_k^k}{(\lambda_k-\lambda_1)(\lambda_k-\lambda_2) \ldots (\lambda_k-\lambda_{k-1})}\right]\notag\\
					&=&	F_{k+1}^{(k)}-F_k^{(k)}.\notag
		\end{eqnarray} 
	In a similar fashion, the coefficients of $t_{i}$s, for $i=1, \ldots, k-1$, were obtained and listed as follows:
		\begin{eqnarray}
			t_{k-1}&:&F_k^{(k)}\notag\\
			t_{k-2}&:&F_{k+1}^{(k)}-F_k^{(k)}\notag\\
			t_{k-3}&:&F_{k+2}^{(k)}-F_{k+1}^{(k)}-F_k^{(k)}\notag\\
			& & \hspace{10mm} \vdots \hspace{10mm} \vdots \notag\\
			t_2&:& F_{2k-3}^{(k)}-F_{2k-4}^{(k)}-...-F_{k+1}^{(k)}-F_k^{(k)}\notag\\
			t_1&:& F_{2k-2}^{(k)}-F_{2k-3}^{(k)}-...-F_{k+1}^{(k)}-F_k^{(k)}\notag\\
			t_0&:& F_{2k-1}^{(k)}-F_{2k-2}^{(k)}-...-F_{k+1}^{(k)}-F_k^{(k)}\notag
		\end{eqnarray}
	By using the definition of $k$-generalized-Fibonacci, 
		\begin{eqnarray}
			F_k^{(k)}&=& F_{k-1}^{(k)}+F_{k-2}^{(k)}+\ldots+F_2^{(k)}+F_1^{(k)}+F_0^{(k)}\notag\\
			F_{k+1}^{(k)}&=& F_{k}^{(k)}+F_{k-1}^{(k)}+\ldots+F_3^{(k)}+F_2^{(k)}+F_1^{(k)}\notag\\
			F_{k+2}^{(k)}&=& F_{k+1}^{(k)}+F_{k+1}^{(k)}+\ldots+F_4^{(k)}+F_3^{(k)}+F_2^{(k)}\notag\\
			& & \hspace{10mm} \vdots \hspace{10mm} \vdots \notag\\
			F_{2k-3}^{(k)}&=&F_{2k-4}^{(k)}+F_{2k-5}^{(k)}+\ldots+F_{k-1}^{(k)}+F_{k-2}^{(k)}+F_{k-3}^{(k)}\notag\\
			F_{2k-2}^{(k)}&=&F_{2k-3}^{(k)}+F_{2k-4}^{(k)}+\ldots+F_k^{(k)}+F_{k-1}^{(k)}+F_{k-2}^{(k)}\notag\\
			F_{2k-1}^{(k)}&=&F_{2k-2}^{(k)}+F_{2k-3}^{(k)}+\ldots+F_{k+1}^{(k)}+F_k^{(k)}+F_{k-1}^{(k)}\notag
		\end{eqnarray}
	So by substitution, the coefficients now become,
		\begin{eqnarray}
			t_{k-1}&:&F_{k-1}^{(k)}+F_{k-2}^{(k)}+\ldots+F_2^{(k)}+F_1^{(k)}+F_0^{(k)}\notag\\
			t_{k-2}&:&F_{k-1}^{(k)}+F_{k-2}^{(k)}+\ldots+F_2^{(k)}+F_1^{(k)}\notag\\
			t_{k-3}&:&F_{k-1}^{(k)}+F_{k-2}^{(k)}+\ldots+F_2^{(k)}\notag\\
			& & \hspace{10mm} \vdots \hspace{10mm} \vdots \notag\\
			t_2&:& F_{k-1}^{(k)}+F_{k-2}^{(k)}+F_{k-3}^{(k)}\notag\\
			t_1&:& F_{k-1}^{(k)}+F_{k-2}^{(k)}\notag\\
			t_0&:& F_{k-1}^{(k)}\notag
		\end{eqnarray}
	But these values are the initial terms of our $k$-generalized-Fibonacci sequence, $F_0^{(k)}=F_1^{(k)}=...=F_{k-2}^{(k)}=0$ and $F_{k-1}^{(k)}=1$. Hence, we conclude that the coefficients of $t_0, t_1, \ldots, t_{k-2}, t_{k-1}$ are all 1.
		
		\item \textbf{Inductive Step:}
		
	Now, let us suppose that $\mathfrak{F}_n^{(k)}$ is true for $n=k,k+1,\ldots,l-1$ until $n=l$, WNTS that $\mathfrak{F}_n^{(k)}$ is also true for $n=l+1$, that is,
			\begin{equation}
				\mathfrak{F}_{l+1}^{(k)}=\sum_{m=1}^{k} \left[\displaystyle \frac{\displaystyle\sum_{p=1}^{k}\left(\lambda_m^{k-p}-\displaystyle\sum_{i=1}^{k-p}\lambda_m^{(k-p)-i}\right)t_{p-1}}{\displaystyle \prod_{j=1, j\neq m}^{k}(\lambda_m - \lambda_j)} \right] \lambda_m^{l+1} \notag
			\end{equation}
		For $n=l$, we have 
			\begin{equation}
				\mathfrak{F}_l^{(k)}=\sum_{m=1}^{k} \left[\displaystyle \frac{\displaystyle\sum_{p=1}^{k}\left(\lambda_m^{k-p}-\displaystyle\sum_{i=1}^{k-p}\lambda_m^{(k-p)-i}\right)t_{p-1}}{\displaystyle \prod_{j=1, j\neq m}^{k}(\lambda_m - \lambda_j)} \right] \lambda_m^l
			\end{equation}
		Adding $\mathfrak{F}_{l-(k-1)}^{(k)}, \mathfrak{F}_{l-(k-2)}^{(k)}, \ldots,\mathfrak{F}_{l-2}^{(k)}, \mathfrak{F}_{l-1}^{(k)}$ both sides of (24), we have the equation
				\begin{equation}
				\mathfrak{F}_{l+1}^{(k)} = \displaystyle\sum_{m=1}^{k} \left[\frac{\displaystyle\sum_{p=1}^{k}\left(\lambda_m^{k-p}-\displaystyle\sum_{i=1}^{k-p}\lambda_m^{(k-p)-i}\right)t_{p-1}}{\displaystyle \prod_{j=1, j\neq m}^{k}(\lambda_m - \lambda_j)} \right] \lambda_m^l +\mathfrak{F}_{l-(k-1)}^{(k)}+\mathfrak{F}_{l-(k-2)}^{(k)}+ \ldots +\mathfrak{F}_{l-1}^{(k)}
				\end{equation}
		Note that the LHS of (25) simplifies to $\mathfrak{F}_{l+1}^{(k)}$ by the recursive definition of $k$-generalized Fibonacci sequence. Simplifying the RHS of (25),

{\small	\begin{eqnarray}
			\mathfrak{F}_{l+1}^{(k)} &=&\sum_{m=1}^{k} \left[\displaystyle\frac{\displaystyle\sum_{p=1}^{k}\left(\lambda_m^{k-p}-\displaystyle\sum_{i=1}^{k-p}\lambda_m^{(k-p)-i}\right)t_{p-1}}{\displaystyle \prod_{j=1, j\neq m}^{k}(\lambda_m - \lambda_j)} \right] \lambda_m^l + \sum_{m=1}^{k} \left[\displaystyle\frac{\displaystyle\sum_{p=1}^{k}\left(\lambda_m^{k-p}-\displaystyle\sum_{i=1}^{k-p}\lambda_m^{(k-p)-i}\right)t_{p-1}}{\displaystyle \prod_{j=1, j\neq m}^{k}(\lambda_m - \lambda_j)} \right] \lambda_m^{l-1} \notag \\
			& & + \ldots + \sum_{m=1}^{k} \left[\displaystyle\frac{\displaystyle\sum_{p=1}^{k}\left(\lambda_m^{k-p}-\displaystyle\sum_{i=1}^{k-p}\lambda_m^{(k-p)-i}\right)t_{p-1}}{\displaystyle \prod_{j=1, j\neq m}^{k}(\lambda_m - \lambda_j)} \right] \lambda_m^{l-(k-1)}\notag\\
			&=&\sum_{m=1}^{k} \left[\displaystyle\frac{\displaystyle\sum_{p=1}^{k}\left(\lambda_m^{k-p}-\displaystyle\sum_{i=1}^{k-p}\lambda_m^{(k-p)-i}\right)t_{p-1}}{\displaystyle \prod_{j=1, j\neq m}^{k}(\lambda_m - \lambda_j)} \right][\lambda_m^l+\lambda_m^{l-1}+ \ldots + +\lambda_m^{l-(k-1)}]\notag\\
			\mathfrak{F}_{l+1}^{(k)} &=&\sum_{m=1}^{k} \left[\displaystyle\frac{\displaystyle\sum_{p=1}^{k}\left(\lambda_m^{k-p}-\displaystyle\sum_{i=1}^{k-p}\lambda_m^{(k-p)-i}\right)t_{p-1}}{\displaystyle \prod_{j=1, j\neq m}^{k}(\lambda_m - \lambda_j)} \right](\lambda_m^l)[1+\lambda_m^{-1}+ \ldots +\lambda_m^{-(k-1)}]
		\end{eqnarray}}
	
		But $\lambda_m$, $m=1, 2, ..., k$, is a root of the characteristic polynomial, $P(\lambda) = \lambda^k - \lambda^{k-1}-\ldots -\lambda_m - 1$. So, 
					\begin{eqnarray}
						\lambda^k_m - \lambda^{k-1}_m - \ldots - \lambda_m - 1 = 0 
						&\Longleftrightarrow&  \lambda^k_m = \lambda^{k-1}_m  + \lambda^{k-2}_m \ldots + 1 \notag\\ 
						&\Longleftrightarrow&  \lambda_m=1+\lambda^{-1}_m+ \ldots +\lambda^{-(k-1)}_m
					\end{eqnarray}
		By substituting (27) in (26), we have
		\begin{eqnarray}
			\mathfrak{F}_{l+1}^{(k)} &=&\displaystyle\sum_{m=1}^{k} \left[\displaystyle\frac{\displaystyle\sum_{p=1}^{k}\left(\lambda_m^{k-p}-\displaystyle\sum_{i=1}^{k-p}\lambda_m^{(k-p)-i}\right)t_{p-1}}{\displaystyle \prod_{j=1, j\neq m}^{k}(\lambda_m - \lambda_j)} \right](\lambda_m^l)(\lambda_m) \notag\\
			&=& \displaystyle\sum_{m=1}^{k} \left[\displaystyle\frac{\displaystyle\sum_{p=1}^{k}\left(\lambda_m^{k-p}-\displaystyle\sum_{i=1}^{k-p}\lambda_m^{(k-p)-i}\right)t_{p-1}}{\displaystyle \prod_{j=1, j\neq m}^{k}(\lambda_m - \lambda_j)} \right](\lambda_m^{l+1})
		\end{eqnarray}
		as desired. Hence, we have shown that $\mathfrak{F}_{n}^{(k)}$ is also true for $n=l+1$. Therefore, the formula is true for $n\geq k$.
\end{enumerate}
\end{proof}

\section*{Conclusion}

\hspace{5.5mm} Through concepts involving eigenvectors and eigenvalues of an operator, we have shown the formula for the $nth$ term of some variations of the Fibonacci sequence. Recognizing patterns could be used to develop general ones. In this paper, a formula for finding the $nth$ term of $k$-generalized Fibonacci-like sequence was developed and proven via Strong Induction. Other sequences could be solved using the method presented.

\pagebreak
\section*{References}

Bacani, J.B., Rabago, J.F.T. (2015). On Generalized Fibonacci Numbers. \textit{Hikari, Ltd.,} 

		\hspace{5mm} \textit{Applied Mathematical Sciences} arXiv:1503.05305v1
		\vspace{2mm}
		\\
Dresden, G.P.B (2014). A Simplified Binet Formula for $k$-Generalized Fibonacci Numbers.

		\hspace{5mm}\textit{Journal of Integer Sequences}, Vol. 17. Article 14.4.7.
		\vspace{2mm}
		\\
Koshy, T. (2007). Fibonacci and Lucas Numbers. \textit{Elementary Number Theory with} 

		\hspace{5mm} \textit{Applications, Second Edition}. Academic Press, an imprint of Elsevier, Inc,
		
		\hspace{5mm} 30 Corporate Drive, Suite 400, Burlington, MA 01803, USA
		\vspace{2mm}
		\\
Kuhapatanakul, K., Sukruan, L. (2014). n-Tribonacci Triangles and Application. \textit{International}

		\hspace{5mm} \textit{Journal of Mathematical Education in Science and Technology}, 45, 1068–1075.
		\vspace{2mm}
		\\
Meinke, A.M. (2011). \textit{Fibonacci Numbers and Associated Matrices}.
		\vspace{2mm}
		\\
Kuhapatanakul, K., Anantakitpaisal, A. (2017). The k-nacci triangle and applications.

		\hspace{5mm} \textit{Cogent Mathematics, 4:1333293}
		\vspace{2mm}
		\\
Natividad, L., Policarpio, P. (2013). A Novel Formula in Solving Tribonacci-like Sequence.

		\hspace{5mm} \textit{International Center for Scientific Research and Studies (ICSRS) Publication}, Gen.
		
		\hspace{5mm} Math. Notes, Vol. 17, No. 1, July, 2013, pp. 82-87. 
		\vspace{2mm}
		\\
Noe, T., Piezas, T. and Weisstein, E. (2012). \textit{Tribonacci Number}.

		\hspace{5mm} Available at: http://mathworld.wolfram.com/ TribonacciNumber.html
		
		\hspace{5mm} (Accessed: 12 August 2019)
		\vspace{2mm}
		\\
Rosen, K.H. (2012). Sequences and Summations. \textit{Discrete Mathematics and Its Applications,}

		\hspace{5mm} \textit{Seventh Edition}. McGraw-Hill, a business unit of The McGraw-Hill Companies, Inc.
		
		\hspace{5mm} 1221 Avenue of the Americas, New York, NY 10020
		\vspace{2mm}
		\\
Weisstein, E.W. \textit{Vieta's Formulas}. Available at: http://mathworld.wolfram.com/

		\hspace{5mm} VietasFormulas.html (Accessed: 15 September 2019) 
		\vspace{2mm}
		\\
Wolfram, D.A. (1998). Solving Generalized Fibonacci Recurrences. \textit{The Fibonacci Quarterly},

		\hspace{5mm} Vol. 36, 129-145.
		\vspace{2mm}
		\\
Wong, C. K., Maddocks, T. W. (1975). A Generalized Pascal’s Triangle. \textit{The Fibonacci}

		\hspace{5mm} \textit{Quarterly}, Vol. 13 No. 2, 134–136.

\end{document}